\documentclass{imanum}
\usepackage{graphicx}

\jno{drnxxx}
\received{1 July 2016}

\begin{document}

\title{Invariance Preserving Discretization Methods of Dynamical Systems}

\author{%
{\sc
Zolt\'{a}n Horv\'{a}th\thanks{Email: horvathz@sze.hu}, 
} \\[2pt]
             Department of Mathematics and Computational Sciences,  
             Sz\'{e}chenyi Istv\'{a}n University\\
              9026 Gy\H{o}r, Egyetem t\'{e}r 1, Hungary\\[6pt]
{\sc and}\\[6pt]
{\sc Yunfei Song\thanks{Corresponding author. Email: songyunfei1986@gmail.com}
and
Tam\'{a}s Terlaky\thanks{Email: terlaky@lehigh.edu}}\\[2pt]
Department of Industrial and Systems Engineering,
           Lehigh University\\
           200 West Packer Avenue, Bethlehem, PA, 18015-1582
}
\shortauthorlist{Zolt\'{a}n Horv\'{a}th, Yunfei Song and Tam\'{a}s Terlaky}

\maketitle

\begin{abstract}
{In this paper, we consider  local and uniform invariance preserving steplength thresholds on a set when a discretization method is applied to a linear or nonlinear dynamical system. 
For the forward or backward Euler method, the existence of local and uniform invariance preserving steplength thresholds is proved when the invariant sets are polyhedra, ellipsoids, or Lorenz cones. Further, we also quantify the steplength thresholds of the backward Euler methods on these sets for linear dynamical systems. 
Finally, we present our main results on the existence of uniform invariance preserving steplength threshold of general discretization methods on  general convex sets, compact sets, and  proper cones  both for  linear and nonlinear dynamical systems.}
{invariant set; invariance preserving; discretization method; compact set; proper cone.}
\end{abstract}

\section{Introduction}
\label{sec;introduction}

A set $\mathcal{S}$ is referred to as a \emph{positively invariant set} for continuous and discrete dynamical systems if  the starting state
of the dynamical system belongs to $\mathcal{S}$ implies that all the forward
states remain in $\mathcal{S}$. This concept has extensive applications in dynamical system  and control theory (see e.g., \cite{Blanchini, mariani, boyd, luen}). 
The popular candidate sets for positively  invariant sets, which are usually studied for continuous and discrete systems, are
polyhedra (see \cite{caste, aless, dorea}), ellipsoids (see \cite{boyd}), and Lorenz
cones (see \cite{loewy, stern, Vander}). The popularity of these special sets is due to their nice properties and the fact that they are widely used in modeling important applications. A unified approach to derive  sufficient and necessary conditions under which a set is a positively invariant set is presented in \cite{song1}.

 In practice, continuous dynamical systems are often approximated by discrete dynamical systems,
e.g., when we compute the numerical solution of a differential
equation. By using certain disretization methods, we should preserve as many
properties of the continuous  dynamical system as possible, in addition to
the requirement to have small approximation error. Thus, if the
continuous  dynamical systems has a positively invariant set, then the same set should also be  a positively invariant set for the corresponding  discrete system which is obtained by using the discretization method.  Such a discretization method is called \emph{invariance preserving} for the dynamical system on the positively invariant set.

In this paper, our focus is  to find conditions, in particular  steplength thresholds for the discretization methods, such that the considered discretization method is invariance preserving for the given linear or nonlinear dynamical system. This topic is of great interest in the fields of dynamical systems, partial differential equations, and control theory.
A basic result is presented in \cite{bolleycrouzeix1978}, which considers linear problems
and  invariance preserving on the positive orthant from a perspective of numerical
methods.
For invariance preserving on the positive orthant or polyhedron for Runge-Kutta methods, the reader is refereed to \cite{horva, horvathapnum2005, horvathfarkasconf2006}.
A similar concept named strong stability preserving
(SSP) used in numerical methods is studied in \cite{sspbook, siga2, shu2}. These papers  deal with invariance preserving of general sets and they usually use the assumption that the  Euler methods are invariance preserving with a
steplength threshold $\tau_0.$ Then the uniform invariance preserving steplength threshold  for other advanced numerical methods, e.g., Runge-Kutta methods, is derived in terms of $\tau_0.$
Therefore, to make the results applicable to solve real world problems, this approach requires  to check whether such a positive $\tau_0$
exists for Euler methods. We note  that quantifications of  steplength thresholds for some classes discretization  methods on a polyhedron are studied in \cite{song3}.

In this paper, basic concepts and theorems are introduced in Section \ref{sec21}. In Section \ref{sec:single} first we prove that for the forward Euler method, a local invariance preserving steplength threshold exists for a given polyhedron when a linear dynamical system is considered.  For the backward Euler method we prove that a local steplength threshold exists for polyhedron, ellipsoid, and Lorenz cone. These proofs are using elementary concepts. We also quantify a valid local steplength threshold for the backward Euler method. 
Second, we prove that a uniform invariance preserving steplength threshold exists for polyhedra when the forward or backward Euler method is applied to linear dynamical systems. For the backward Euler method, we also quantify the optimal uniform steplength threshold. 
In Section \ref{sec:uniform} we first prove that a uniform steplength threshold exists, and also quantify the optimal uniform steplength threshold for ellipsoids or Lorenz cones when the backward Euler method is applied to linear dynamical systems. Moreover, we extend the results about the invariance preserving steplength threshold for the backward Euler method to general proper cones. 
Finally, we present our main results about uniform steplength thresholds. These results are natural extensions from the proofs used to analyze Euler methods. We quantify the optimal uniform steplength threshold of the backward Euler methods for convex sets. We also extend the results of steplength thresholds to general compact sets and proper cones when a general discretization method is applied to linear or nonlinear dynamical systems. In particular, the existence of steplength thresholds depends on a condition that is stronger than local steplength threshold\footnote{In particular, this condition requires   that if $x_k$ is in the set, then  $x_{k+1}$ is in the interior of the set.}. It also depends on the Lipschitz condition when the set is a compact  set or homogenous condition when the set is a proper cone. Our conclusions are summarized in Section \ref{sec:con}.

The main novelty of this paper is establishing the foundation of characterizing  invariance preserving  discretization methods in dynamical systems and differential equations. As mentioned before,  several existing results on invariance preserving of advanced numerical methods, e.g., Runge-Kutta methods, require the existence of a positive steplength threshold for Euler methods. In this paper, we present the results for special classical sets. Our general results about steplength threshold for general discretization methods for linear and nonlinear dynamical systems on convex sets, compact sets, and proper cones not only play an important role in the theoretical perspective, but also show the potential of significant impacts in practice. These general results provide theoretical criteria for the verification of the existence of invariance preserving steplength threshold for discretization methods.  Once the existence is ensured by our results, this also motivates one to further investigate the possibility to find the optimal steplength threshold, which has several advantages in practice. Such advantages include   computational efficiency and smaller size of discrete systems.

\emph{Notation and Conventions.} To avoid unnecessary repetitions,
the following notations and conventions are used in this paper. The $i$-th row of a matrix $G$ is denoted by $G_i^T.$ The interior and the boundary of a set $\mathcal{S}$ is denoted by
int$(\mathcal{S})$ and  $\partial \mathcal{S}$, respectively.
  The index set $\{1,2,...,n\}$ is
denoted by $\mathcal{I}(n).$ A symmetric positive definite and positive semidefinite matrix is denoted by $Q\succ0$ and $Q\succeq0$, respectively.

\section{Preliminaries} \label{sec21}

In this paper, we consider  discrete and continuous linear dynamical
systems which are respectively represented as 
\begin{equation}\label{dyna2}
x_{k+1}=Ax_k,
\end{equation}
\begin{equation}\label{dyna1}
\dot{x}(t)=Ax(t),
\end{equation}
where $A\in \mathbb{R}^{n\times n}$, $x_k,  x(t)\in \mathbb{R}^n$, $t\in
\mathbb{R}$, and $k\in \mathbb{N}$.  Invariant sets for discrete and continuous systems are introduced as follows:

\begin{definition}\label{def1}
Let  $\mathcal{S}\subseteq \mathbb{R}^n$. If $ x_k\in \mathcal{S}$
implies $ x_{k+1}\in \mathcal{S}$, for all
 $k\in \mathbb{N}$, then $\mathcal{S}$ is an invariant  set for
the discrete system (\ref{dyna2}).
\end{definition}

\begin{definition}\label{def2}
Let  $\mathcal{S}\subseteq \mathbb{R}^n$. If $x(0)\in \mathcal{S}$ implies
$x(t)\in\mathcal{S}$,  for all $t\geq0$, then $\mathcal{S}$ is an invariant set for
the continuous system (\ref{dyna1}).
\end{definition}

According to Definition \ref{def1}, we have that if $\mathcal{S}$ is an invariant set for the discrete system (\ref{dyna2}), then $x_0\in\mathcal{S}$ implies that $x_k\in\mathcal{S}$ for all $k\in \mathbb{N}.$ In fact, the sets defined in Definition \ref{def1} and \ref{def2} are conventionally referred to as positively invariant sets, since only the positive time domain is considered. For simplicity, we  call them invariant sets. 

In this paper, some special sets, namely polyhedra, ellipsoids,  and
Lorenz cones are considered as candidate invariant sets for both discrete and continuous systems. We now formally define these sets.  
A \emph{polyhedron} $\mathcal{P}\subseteq \mathbb{R}^n$ can be represented as 
\begin{equation}\label{poly1}
\mathcal{P}=\{x\in \mathbb{R}^n\,|\,Gx\leq b\},
\end{equation}
where $G\in \mathbb{R}^{m\times n} $, $ b\in \mathbb{R}^m$.
Equivalently, 
\begin{equation}\label{poly2}
\mathcal{P}=\Big\{x\in \mathbb{R}^n\,|\,x=\sum_{i=1}^{\ell_1}\theta_i
x^i+\sum_{j=1}^{\ell_2}\hat{\theta}_j\hat{x}^j,~
\sum_{i=1}^{\ell_1}\theta_i=1, \theta_i\geq0, \hat{\theta}_j\geq
0\Big\},
\end{equation}
where $x^1,...,x^{\ell_1},\hat{x}^1,...,\hat{x}^{\ell_2}\in
\mathbb{R}^n$. 
An \emph{ellipsoid} $\mathcal{E}\subseteq \mathbb{R}^n$ centered at the
origin can be represented as 
\begin{equation}\label{elli}
\mathcal{E}=\{x\in\mathbb{ R}^n \,|\, x^TQx\leq 1\},
\end{equation}
where  $Q\in \mathbb{R}^{n\times n}$ and $Q\succ0$. 
A \emph{Lorenz cone}\footnote{A Lorenz cone also refers to as an
ice cream cone, or a second order cone.} $\mathcal{C_L}\subseteq \mathbb{R}^n$
with its vertex at the origin  can be represented as 
\begin{equation}\label{ellicone}
\mathcal{C_L}=\{x\in \mathbb{R}^n\,|\,x^TQx\leq 0,\, x^TQu_n\leq0\},
\end{equation}
where $Q\in \mathbb{R}^{n\times n}$ is a symmetric
matrix and ${\rm
inertia}\{Q\}=\{n-1,0,1\}$.

The following definition introduces the concepts of invariance preserving and steplength threshold. 
\begin{definition}\label{invpre}
Assume a set $\mathcal{S}$ is an invariant set for the continuous system (\ref{dyna1}), and a discretization method is applied to the continuous system to yield a discrete system.  
\begin{itemize}
\item For a given $x_k\in \mathcal{S}$, if there exists a $\tau(x_k)>0,$ such that $x_{k+1}\in\mathcal{S}$ for $\Delta t\in [0,\tau(x_k)]$, where $x_{k+1}$ is obtained by using the discretization method, then the discretization method is \textbf{locally invariance preserving at $x_k$}, and $\tau(x_k)$ is a \textbf{local invariance preserving steplength threshold}  for this discretization method \textbf{at $x_k$}.
\item If there exists a $\tau>0,$ such that $\mathcal{S}$ is also an invariant set for the discrete system for any steplength $\Delta t\in[0,\tau]$, then the discretization method is \textbf{uniformly invariance preserving on $\mathcal{S}$} and $\tau$ is a \textbf{uniform  invariance preserving  steplength threshold}  for this discretization method \textbf{on $\mathcal{S}$}.
\end{itemize}
\end{definition}

The forward and backward Euler methods are simple first order discretization methods that are usually
applied to  solve ordinary differential equations numerically with
initial conditions.  The forward Euler method, which is an
explicit method, is conditionally stable. On the other hand, the backward Euler
method, which is an implicit method, is unconditionally stable, see, e.g., \cite{high2}. The
two Euler methods for the continuous system $\dot{x}(t)=Ax(t)$ are
described as follows:
\begin{enumerate}
  \item {Forward Euler Method:  }\begin{equation}\label{euler1}
 \frac{x_{k+1}-x_k}{\Delta t}=Ax_k
\Rightarrow x_{k+1}=(I+\Delta t A)x_k.
\end{equation}
  \item {Backward Euler Method:  }\begin{equation}\label{euler2}
\frac{x_{k+1}-x_k}{\Delta t}=Ax_{k+1}\Rightarrow x_{k+1}=(I-\Delta t
A)^{-1}x_k.
\end{equation}
\end{enumerate}

Now we consider the effects of the Euler methods on the
continuous system, i.e., given a vector $x_k$ in
$ \mathcal{S},$ we investigate conditions that ensure that
$x_{k+1}$ obtained by (\ref{euler1}) or (\ref{euler2}) is also in $
\mathcal{S}$. A geometric interpretation of the forward Euler
method is that $x_{k+1}$  is  on the tangent line of $x(t)$ at
boundary  point $x_k$.  For
a convex set $\mathcal{S}$, it is well known that the tangent
space at $x_k$ on the boundary of $\mathcal{S}$ is a supporting
hyperplane to $\mathcal{S}$, see e.g., \cite{rock}.  Figure \ref{fig1} illustrates the effects of
the Euler methods on two classes of trajectories. In these two
cases, the convex sets include the trajectory on its boundary, and
include the region above the curves. The left subfigure of Figure
\ref{fig1} shows that the forward and backward Euler methods lead
the discrete steps direct outside and inside the convex set,
respectively. The right subfigure of Figure \ref{fig1} shows that
the discrete steps for both Euler methods are on the boundary.

\begin{figure}[h]
    \centering
    \includegraphics[width=0.4\textwidth]{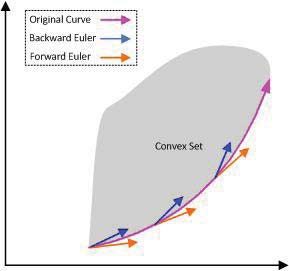}~~~~~~~
    \includegraphics[width=0.4\textwidth]{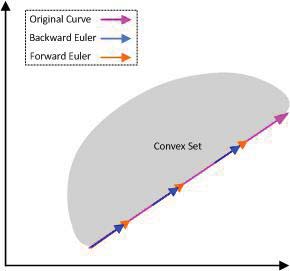}
    \caption{The left figure illustrates when $x(t)$ is a curve, the right figure illustrates when $x(t)$ is a line.}
    \label{fig1}
\end{figure}

\section{Local Steplength Threshold}\label{sec:single}
In this section, we prove the existence of an invariance preserving local steplength threshold   when the invariant sets are polyhedra, ellipsoids, and Lorenz cones.

\subsection{Existence of Local Steplength Threshold}
We first consider polyhedral sets and the forward and backward Euler methods for linear systems. 
\begin{lemma}\label{lemma20}
Assume that a polyhedron $\mathcal{P}$, given as in (\ref{poly1}), is an invariant set for the continuous system (\ref{dyna1}), and $x_k\in \mathcal{S}$.
Then there exists a $\tau(x_k)>0$, such that
$x_{k+1}\in\mathcal{P}$ for all $\Delta t\in[0,\tau(x_k)]$, where $x_{k+1}$ is obtained by the forward Euler method
(\ref{euler1}).
\end{lemma}
\begin{proof}
 Since $\text{int}(\mathcal{P})$ is an open set, we have that the statement is true for $x_k\in\text{int}(\mathcal{P})$.  For $x_k\in \partial \mathcal{P}$, by Nagumo's Theorem, see e.g., \cite{song1, nagu}, the statement is also true. 
\end{proof}

In fact, the proof of Lemma \ref{lemma20} is also applicable for nonlinear systems, thus a similar conclusion about the local steplength threshold can be obtained for nonlinear systems too.  Now we turn our attention to the backward Euler method.

\begin{lemma}\label{lemma201}
Assume that a polyhedron $\mathcal{P}$,  given as in (\ref{poly1}),
is an invariant set for the continuous system (\ref{dyna1}), and $x_k\in \mathcal{P}$. Then there exists a $\tau(x_k)>0$, such that $x_{k+1}\in\mathcal{P}$ for all $\Delta t\in[0,\tau(x_k)]$, where $x_{k+1}$ is obtained
by the backward Euler method (\ref{euler2}).
\end{lemma}
\begin{proof} 
Since $\mathcal{P}$ is an invariant set for the continuous system (\ref{dyna1}), we have $G(e^{At}x)\leq b$
for all $t\geq0.$ By substituting 
$e^{At}=I+At+\tfrac{1}{2}A^2t^2+\cdots$, we have $Gx+tGAx+\tfrac{t^2}{2!}GA^2x+\cdots\leq b$ for all $t\geq0,$ which, for all $t\geq0,$ can be written as
\begin{equation}\label{eq3331}
G_i^Tx+tG_i^TAx+\tfrac{t^2}{2!}G_i^TA^2x+\tfrac{t^3}{3!}G_i^TA^3x+\cdots\leq b_i, \text{ for } i\in \mathcal{I}(n).
\end{equation}
For the backward Euler method we need to prove that for given $Gx_k\leq b$ there exists a $\tau(x_k)>0$, such that $G(I-A\Delta t)^{-1}x_k\leq b$, for $\Delta t\in [0,\tau(x_k)],$ which, by using  $(I-A\Delta
t)^{-1}=I+A\Delta t+A^2\Delta t+\cdots$, is equivalent to  
\begin{equation}\label{eq3332}
G_i^Tx_k+\Delta tG_i^TAx_k+(\Delta t)^2G_i^TA^2x_k+(\Delta t)^3G_i^TA^3x_k+\cdots\leq b_i, \text{ for } i\in \mathcal{I}(n),
\end{equation}
for $\Delta t\in[0, \tau(x_k)].$ For $i\in \mathcal{I}(n)$, we
denote the  bound for $\Delta t$ by $\tau_i(x_k)\geq0,$ such that
(\ref{eq3332}) holds. We have the following
three cases:
\begin{itemize}
\item If $G_i^Tx_k<b_i,$ then $\tau_i(x_k)>0$ due to the fact that int$(\mathcal{P})$ is an open set.

\item If there exists an $\ell\geq1$, such that $G_i^Tx_k=b, G_i^TAx_k=0,...,G_i^TA^{\ell-1}x=0, $ and $G_i^TA^\ell x_k<0$, then according to (\ref{eq3332}), we have 
$\tau_i(x_k) >0$.

\item If neither of the above two cases is true, then we have $G_i^Tx_k=b, G_i^TA^jx_k=0$ for all $j=1,2,...,$ which yields  $\tau_i(x_k)=\infty.$
\end{itemize}
Let $\tau(x_k)=\min_{i\in \mathcal{I}(n)}\{\tau_i(x_k)\}.$ Since $\mathcal{I}(n)=n$ is finite, we have $\tau(x_k)>0.$ Clearly, when $\Delta t\in[0,\tau(x_k)],$ we have $x_{k+1}=(I-A\Delta t)^{-1}x_k\in\mathcal{P}$. The proof is complete. 
\end{proof}

We now consider ellipsoids and Lorenz cones. If the trajectory of the continuous system is on the boundary of a given ellipsoid or Lorenz cone, then according to the fact that the forward Euler method yields the tangent line of the trajectory at the given point $x_k$, we have that the forward Euler method is not invariance preserving for any $\Delta t>0$. Thus, we only consider the backward Euler method for ellipsoids and Lorenz cones.  
\begin{lemma}\label{lemma21}
Assume that an ellipsoid $\mathcal{E}$, given as in (\ref{elli}),
is an invariant set for the continuous system (\ref{dyna1}), and $x_k\in \mathcal{E}$. Then there exists a $\tau(x_k)>0$, such that $x_{k+1}\in\emph{int}(\mathcal{E})$ for all $\Delta t\in[0,\tau(x_k)]$, where $x_{k+1}$ is obtained
by the backward Euler method (\ref{euler2}).
\end{lemma}
\begin{proof}
It is easy to show that  $Ax_k=0$ implies $x_{k+1}=x_k$,
thus we consider the case of  $Ax_k\neq0$. Since $\text{int}(\mathcal{E})$ is an open set, it is trivial to find $\tau(x_k)>0$ for $x_k\in \text{int}(\mathcal{E})$. Thus we consider only the case when $x_k\in\partial\mathcal{E}$, i.e., $x_k^TQx_k=1.$

Since
$\mathcal{E}$ is an invariant set for the continuous system, we have
$x_k^T(e^{At})^TQ(e^{At})x_k\leq1$ for all $t\geq0.$ By substituting 
$e^{At}=I+At+\frac{1}{2}A^2t^2+\mathcal{O}(t^3)$, we have 
\begin{equation*}
x_k^TQx_k+tx_k^T(A^TQ+QA)x_k+t^2(\tfrac{1}{2}x_k^T(A^{2T}+A^2)x_k+(Ax_k)^TQ(Ax_k))+\mathcal{O}(t^3)\leq
1
\end{equation*} 
for all $t\geq0, $ which is, by noting that $x_k^TQx_k=1,$ equivalent to
\begin{equation}\label{eq223}
x_k^T(A^TQ+QA)x_k+t(\tfrac{1}{2}x_k^T(A^{2T}Q+QA^2)x_k+(Ax_k)^TQ(Ax_k))+\mathcal{O}(t^2)\leq
0
\end{equation}
for all  $t\geq0.$ 
If $x_k^T(A^TQ+QA)x_k=0$, then (\ref{eq223}) implies
\begin{equation}\label{eq4876}
\tfrac{1}{2}x_k^T(A^{2T}Q+QA^2)x_k+(Ax_k)^TQ(Ax_k)\leq0.
\end{equation} 
Since $Ax_k\neq0$ and $Q\succ0,$ then $(Ax_k)^TQ(Ax_k)\succ0,$ which, according to (\ref{eq4876}), yields
\begin{equation}\label{eq2211}
\tfrac{1}{2}x_k^T(A^{2T}Q+QA^2)x_k<0.
\end{equation}

For the discrete system obtained by the backward
Euler method, by using $(I-A\Delta
t)^{-1}=I+A\Delta t+A^2\Delta t+\mathcal{O}((\Delta t)^3)$, we have
\begin{equation}\label{eq224}
\begin{split}
x_{k+1}^TQx_{k+1}=&1+\Delta t x_k^T(A^TQ+QA)x_k\\
&+(\Delta
t)^2(x_k^T(A^{2T}+A^2)x_k+(Ax_k)^TQ(Ax_k))+\mathcal{O}((\Delta
t)^3).
\end{split}
\end{equation}
Then we consider the following two cases:
\begin{itemize}
\item If $x_k^T(A^TQ+QA)x_k<0$, then $x_{k+1}^TQx_{k+1}<1$ for sufficiently
small $\Delta t$.

\item  If $x_k^T(A^TQ+QA)x_k=0,$ then, according to (\ref{eq2211}),
we have
\begin{equation*}\label{eq222}
x_k^T(A^{2T}Q+QA^2)x_k+(Ax_k)^TQ(Ax_k)\leq
\tfrac{1}{2}x_k^T(A^{2T}Q+QA^2)x_k<0,
\end{equation*}
i.e., the coefficient of $(\Delta t)^2$ in (\ref{eq224}) is
negative, which yields $x_{k+1}^TQx_{k+1}<1$ for sufficiently
small $\Delta t$.
\end{itemize}
Thus, there exists a $\tau(x_k)>0$ such that  $x_{k+1}\in \text{ int }(\mathcal{E})$ for all
 $\Delta t\in[0,\tau(x_k)]$. The proof is complete.
\end{proof}

Now we are ready to extend the result of Lemma \ref{lemma21} to the case of Lorenz cones. 

\begin{lemma}\label{lemma22}
Assume that a Lorenz cone $\mathcal{C_L}$, given as in (\ref{ellicone}), is
an invariant set for the continuous system (\ref{dyna1}), and $x_k\in \mathcal{C_L}$. Then there exists a $\tau(x_k)>0$, such that $x_{k+1}\in\mathcal{C_L}$ for all $\Delta t\in[0,\tau(x_k)]$, where $x_{k+1}$ is
obtained by the backward Euler method (\ref{euler2}).
\end{lemma}
\begin{proof}
Since  $x_k=0$ implies $x_{k+1}=0$, we  consider only the case of $x_k\neq0.$ The
idea of the proof is similar to the proof of Lemma \ref{lemma21}. Since
inequality (\ref{eq223})  also holds for $\mathcal{C_L}$, we have
$x_k^T(A^TQ+QA)x_k\leq0$. If $x_k^T(A^TQ+QA)x_k=0$, then $(Ax_k)^TQx_k=0$, i.e., the inner
product of $Ax_k$ and $Qx_k$ is 0. This
shows that $Ax_k$ is in the tangent plane of $\mathcal{C_L}$ at
$x_k$, since $Qx_k$ is the normal direction at $x_k$ with respect to $\mathcal{C_L}$. The intersection of the
tangent plane and the cone is a half line, thus we consider the following
two cases:
\begin{itemize}
\item If $Ax_k \in\partial\mathcal{C_L}$, i.e.,
$(Ax_k)^TQ(Ax_k)=0$, then $Ax_k$ is in the
intersection of the cone $\mathcal{C_L}$ and the tangent plane
of cone $\mathcal{C_L}$ at $x_k$. Also, since this
intersection is a half line, we have $Ax_k=\lambda_k x_k$ for some $\lambda_k>0,$ i.e., the vector
$x_k$ is an eigenvector of $A$. Thus, we have  $
x_{k+1}=x_k+\lambda_k\Delta t x_k+(\lambda_k\Delta t)^2
x_k+\cdots=\tfrac{x_k}{1-\lambda_k \Delta t}, $ for $\Delta t\in[0,\lambda_k^{-1}),$ which implies
that $x_{k+1}\in \partial \mathcal{C_L}$ for all $\Delta t\in[0,\lambda_k^{-1})$.

\item  If $Ax_k \notin\partial\mathcal{C_L}$, i.e.,
$(Ax_k)^TQ(Ax_k)>0,$ then the rest of the proof is analogous to the proof
of Lemma \ref{lemma21}, which leads to the conclusion that
$x_{k+1}\in \text{int}(\mathcal{C_L}).$ 
\end{itemize}
Thus, there exists a $\tau(x_k)>0$, such that $x_{k+1}\in \mathcal{C_L}$ for all $\Delta t\in[0,\tau(x_k)]$. The proof is complete.
\end{proof}

\subsection{Computation of Local Steplength Threshold}
Lemma \ref{lemma21} and \ref{lemma22} show the existence of a valid steplength threshould such that $x_{k+1}$ obtained by
the backward Euler method is also in the invariant set. In fact, given $x_k\in \mathcal{E}$
(or $\mathcal{C_L}$), we can quantify the steplength threshould. 

For simplicity we consider only the case of  $\mathcal{E}$. To ensure $x_{k+1}\in\mathcal{E}$, we need 
\begin{equation}\label{eq41}
 x_{k+1}^TQx_{k+1}=x_k^TQx_k+\Delta t x_k^T(A^TQ+QA)x_k+\cdots\leq 1.
\end{equation}
We introduce the following notations to represent the sum
of the remaining infinitely many terms starting from the first,
second, and third term in (\ref{eq41}), respectively.
\begin{equation*}
\begin{split}
\sigma_1&=\Delta t x_k^T(A^TQ+QA)x_k+\sigma_2,\\
\sigma_2&=(\Delta
t)^2x_k^T(A^{2T}Q+A^TQA+QA^2)x_k+\sigma_3,\\
\sigma_3&=(\Delta t)^3x_k^T(A^{3T}Q+A^{2T}QA+A^TQA^2+QA^3)x_k+\mathcal{O}((\Delta t)^4).\\
\end{split}
\end{equation*}
Now we use the fact that $\|M+N\|\leq\|M\|+\|N\|$ and
$\|MN\|\leq\|M\|\|N\|$, where $M$ and $N$ are matrices of appropriate dimensions. For simplicity we denote
$\|A\|\Delta t$ by $\tilde{\alpha}$. We can bound $\sigma_1$ as

\begin{equation}\label{eq325}
|\sigma_1|\leq
\|Q\|\|x_k\|^2\left(2\tilde{\alpha}+3\tilde{\alpha}^2+\cdots\right)
=\|Q\|\|x_k\|^2\tfrac{2\tilde{\alpha}-\tilde{\alpha}^2}{(1-\tilde{\alpha})^2},
\end{equation}
where (\ref{eq325}) holds when $\tilde{\alpha}\leq 1$, i.e.,
$\Delta t\leq \tfrac{1}{\|A\|}. $ Similarly, for $\sigma_2$ and
$\sigma_3$, we have
\begin{equation}\label{eq326}
|\sigma_2|\leq
\|Q\|\|x_k\|^2\tfrac{3\tilde{\alpha}^2-2\tilde{\alpha}^3}{(1-\tilde{\alpha})^2},
\text{ and } |\sigma_3|\leq
\|Q\|\|x_k\|^2\tfrac{4\tilde{\alpha}^3-3\tilde{\alpha}^4}{(1-\tilde{\alpha})^2},
\end{equation}
where  $\Delta t\leq \frac{1}{\|A\|}. $ We now consider the
following three cases.

1). If $x_k^TQx_k:=\delta_1 <1$, i.e., $x_k\in\text{int}(\mathcal{E}),$  then to ensure that (\ref{eq41}) holds, we let
$|\sigma_1|\leq 1-\delta_1$, which is true when
\begin{equation}\label{eq331}
\tfrac{2\tilde{\alpha}-\tilde{\alpha}^2}{(1-\tilde{\alpha})^2}\leq\tfrac{
1-\delta_1}{\|Q\|\|x_k\|^2}, \text{ i.e., } \Delta
t\leq
\tfrac{1}{\|A\|}(1-\tfrac{1}{\sqrt{1+\beta_1}}):=
\gamma_1,
\end{equation}
where $\beta_1=(1-\delta_1)\|Q\|^{-1}\|x_k\|^{-2}>0. $

2). If $x_k^TQx_k=1,$ i.e., $x_k\in\partial\mathcal{E},$ and $x_k^T(A^TQ+QA)x_k:=-\delta_2<0,$ then to ensure that (\ref{eq41}) holds,  we let $|\sigma_2|\leq
  \delta_2\Delta t$, which is true when
\begin{equation}\label{eq332}
\tfrac{3\tilde{\alpha}^2-2\tilde{\alpha}^3}{(1-\tilde{\alpha})^2}\leq
\tfrac{\delta_2\|A\|\Delta t}{\|A\|\|Q\|\|x_k\|^2}, \text{ i.e., } \Delta
t\leq\tfrac{1}{\|A\|}(\tfrac{2\beta_2+3-\sqrt{4\beta_2+9}}{2\beta_2+4}):=
\gamma_2,
  \end{equation}
  where $\beta_2=\delta_2\|A\|^{-1}\|Q\|^{-1}\|x_k\|^{-2}>0.$

3). If neither of the previous two cases hold, then according to
  (\ref{eq222})
  we have $x_k^T(A^{2T}Q+A^TQA+QA^2)x_k:=-\delta_3<0$.
  Then to ensure that 
  (\ref{eq41}) holds, we let $|\sigma_3|\leq \delta_3(\Delta t)^2$, which
  is true when
  \begin{equation}\label{eq333}
\tfrac{4\tilde{\alpha}^3-3\tilde{\alpha}^4}{(1-\tilde{\alpha})^2}\leq
\tfrac{\delta_3(\|A\|\Delta t)^2}{\|A\|^2\|Q\|\|x_k\|^2}, \text{ i.e., }  \Delta
t\leq\tfrac{1}{\|A\|}(\tfrac{\beta_3+2-\sqrt{\beta_3+4}}{\beta_3+3}):=
\gamma_3,
  \end{equation}
  where $\beta_3=\delta_3\|A\|^{-2}\|Q\|^{-1}\|x_k\|^{-2}>0.$

Clearly, we have $\gamma_1, \gamma_2, \gamma_3\in(0, \frac{1}{\|A\|}),$ which is consistent with 
conditions  (\ref{eq325}) and (\ref{eq326}). The analysis for a
cone $\mathcal{C_L}$ can be done analogously. The results are
summarized in the following lemma.
\begin{lemma}\label{lemma33}
Assume that an ellipsoid $\mathcal{E}$, given as in (\ref{elli})
(or a Lorenz cone $\mathcal{C_L}$, given as in (\ref{ellicone})), is an invariant set for
the continuous system (\ref{dyna1}),
and $x_k\in \mathcal{E}$ (or $x_k\in \mathcal{C_L}$). Then
$x_{k+1}\in\mathcal{E}$ (or  $\mathcal{C_L}$), where $x_{k+1}$ is obtained by the backward Euler method
(\ref{euler2}) with
\begin{itemize}
  \item $\Delta t \in[0,\gamma_1)$, if $x_k\in $ \emph{int}$(\mathcal{E})$
  (or \emph{int}$(\mathcal{C_L})$),
  \item $\Delta t \in[0,\gamma_2)$, if $x_k\in \partial\mathcal{E}$ (or $\partial \mathcal{C_L}$)
  and $(Ax_k)^TQx_k<0,$
  \item $\Delta t \in [0, \gamma_3)$, if $x_k\in \partial\mathcal{E}$ (or $\partial \mathcal{C_L}$)
  and $(Ax_k)^TQx_k=0,$
\end{itemize}
where $\gamma_1, \gamma_2$ and $\gamma_3$ are defined as in
(\ref{eq331}), (\ref{eq332}), and (\ref{eq333}), respectively.
\end{lemma}

Note that $\gamma_1, \gamma_2$, or $\gamma_3$ might be
quite small. Let us consider an ellipsoid as an example. If $x_k$
is sufficiently close to the boundary, then we have $x_k^TQx_k:=\delta_1\approx 1$, which yields that $\gamma_1\approx0$.

We now present two simple examples, in which the
forward Euler method is not invariance preserving, while the
backward Euler method is invariance preserving.

\begin{example}\label{exmp1}
Consider the  ellipsoid
$\mathcal{E}=\{(\xi,\eta)\;|\;\xi^2+\eta^2\leq1\}$ and the system
$\dot{\xi}=-\eta, \dot{\eta}=\xi.$
\end{example}

The solution of this system is $\xi(t)=\alpha\cos t+\beta\sin t$ and
$ \eta(t)=\alpha \sin t-\beta \cos t$, where $\alpha, \beta$ are two
parameters that depend on the initial condition. The solution
trajectory is a circle, thus $\mathcal{E}$ is an invariant set for the system. 
If we apply the forward Euler method, the discrete system
is
$ \xi_{k+1}=\xi_k-\Delta t \eta_k, ~~ \eta_{k+1}=\Delta t \xi_k +\eta_k. $
Thus, we obtain $\xi_{k+1}^2+\eta_{k+1}^2=(1+(\Delta
t)^2)(\xi_k^2+\eta_k^2)> \xi_k^2+\eta_k^2$, which yields
$(\xi_{k+1},\eta_{k+1})\notin \mathcal{E}$ for every $\Delta t>0$
when $(\xi_{k},\eta_{k})\in
\partial\mathcal{E}$. If we apply the backward Euler method, the
discrete system is $\xi_{k+1}=\frac{1}{1+(\Delta t)^2}(\xi_k-\Delta
t \eta_k), ~~ \eta_{k+1}=\frac{1}{1+(\Delta t)^2}(\Delta t \xi_k +
\eta_k).$
Thus we obtain that $\xi_{k+1}^2+\eta_{k+1}^2=\frac{1}{1+(\Delta
t)^2}(\xi_{k}^2+\eta_{k}^2)\leq \xi_{k}^2+\eta_{k}^2$, which yields
$(\xi_{k+1},\eta_{k+1})\in \mathcal{E}$ for every $\Delta t\geq0,$
when $(\xi_{k},\eta_{k})\in \mathcal{E}$.

\begin{example}\label{exmp2}
Consider the Lorenz cone
$\mathcal{C_L}=\{(\xi,\eta,\zeta)\;|\;\xi^2+\eta^2\leq \zeta^2,
\zeta\geq0\}$ and the system
$\dot{\xi}=\xi-\eta, \dot{\eta}=\xi+\eta, \dot{\zeta}=\zeta.$
\end{example}

The solution of the system is  $\xi(t)=e^t(\alpha\cos t+\beta\sin t),$
$\eta(t)=e^t(\alpha\sin t-\beta\cos t)$ and $\zeta(t)=\gamma e^t,$
where $\alpha, \beta, \gamma$ are three parameters depending on the
initial condition. It is easy to show that $\mathcal{C_L}$ is an invariant set for the 
system.  If we apply the forward Euler
method, the discrete system is
\begin{equation*}\label{eq311}
\scriptsize{\begin{pmatrix}
  \begin{array}{c}
    \xi_{k+1}\\
    \eta_{k+1}\\
    \zeta_{k+1}\\
  \end{array}
\end{pmatrix}}= \left(
  \begin{array}{ccc}
    1+\Delta t & -\Delta t &0 \\
    \Delta t  & 1+\Delta t &0 \\
0 & 0 &1+\Delta t\\
  \end{array}
\right) \left (
  \begin{array}{c}
    \xi_k \\
    \eta_k \\
\zeta_k\\
  \end{array}
\right).
\end{equation*}
However, if we choose any $(\xi_k,\eta_k,\zeta_k)\in
\partial\mathcal{C_L}$, then we have
$(\xi_{k+1},\eta_{k+1},\zeta_{k+1})\notin \mathcal{C_L},$ since
$\xi_{k+1}^2+\eta_{k+1}^2=(1+(1+\Delta t)^2)(\xi_k^2+\eta_k^2)>
(1+\Delta t)^2(\xi_k^2+\eta_k^2)=\zeta_{k+1}^2$, for all $\Delta
t>0$. If we apply the backward Euler method, the discrete system is
\begin{equation*}\label{eq321}
\scriptsize\left(
  \begin{array}{c}
    \xi_{k+1}\\
    \eta_{k+1}\\
    \zeta_{k+1}\\
  \end{array}
\right)=\frac{1}{\omega} \left(
  \begin{array}{ccc}
    (1-\Delta t)^2 & -\Delta t(1-\Delta t) &0 \\
    \Delta t(1-\Delta t)  & (1-\Delta t)^2 &0 \\
0 & 0 &(1-\Delta t)^2+\Delta t^2\\
  \end{array}
\right) \left(
  \begin{array}{c}
    \xi_k \\
    \eta_k \\
    \zeta_k\\
  \end{array}
\right),
\end{equation*}
where $\omega=(1-\Delta t)((1-\Delta t)^2+\Delta t^2).$ If we choose
any $(\xi_k,\eta_k,\zeta_k)\in \mathcal{C_L},$ then we have
$(\xi_{k+1},\eta_{k+1},\zeta_{k+1})\in \mathcal{C_L}$, since
$\xi_{k+1}^2+\eta_{k+1}^2=\frac{1}{\omega^2}((1-\Delta t)^4+\Delta
t^2(1-\Delta t)^2)(\xi_k^2+\eta_k^2)\leq
\frac{1}{\omega^2}((1-\Delta t)^2+\Delta
t^2)^2(\xi_k^2+\eta_k^2)\leq \frac{1}{\omega^2}((1-\Delta
t)^2+\Delta t^2)^2z_k^2=\zeta_{k+1}^2$ for all $\Delta t>0$, and
$\zeta_{k+1}=\frac{1}{\omega}((1-\Delta t)^2+\Delta
t^2)\zeta_k\geq0$ for all $\Delta t\in[0,1).$

\section{Uniform Steplength Threshold}\label{sec:uniform}
In the  analysis of Section \ref{sec:single}, the invariance preserving steplength threshold 
depends on the given $x_k$. However, such steplength threshold may lead inconvenience in
practice, i.e., one has to sequentially modify the value of invariance preserving $\Delta t$  
when $x_k$ is changing. Thus,
it is important and useful to obtain a uniform  steplength threshold for invariance preserving that
depends only on the given invariant set.

\subsection{Uniform Steplength Threshold for Linear Systems}
We first consider polyhedral sets and the forward  Euler method. Note that a similar results for polytope is presented in \cite{blan4}.
\begin{theorem}\label{thema1}
Assume that a polyhedron $\mathcal{P}$ is an invariant set for the
continuous system (\ref{dyna1}). Then there exists $\hat\tau>0$, such
that for every $x_k\in \mathcal{P}$ and $\Delta t\in[0, \hat{\tau}]$,
we have $x_{k+1}\in \mathcal{P}$, where $x_{k+1}$ is obtained by the
forward Euler method (\ref{euler1}), i.e., $\mathcal{P}$ is an invariant set for the discrete system.
\end{theorem}
\begin{proof}
Let us assume that $\mathcal{P}$ is given as in (\ref{poly2}). According to Lemma 3.11 in \cite{song1}, we have that $\mathcal{P}$ is an invariant set for the continuous
system if and only if $Ax^i \in \mathcal{T_P}(x^i)$ and $A\hat {x}^j \in
\mathcal{T_P}(x^i+\hat{x}^j)$ for $i\in \mathcal{I}(\ell_1)$ and $j\in \mathcal{I}(\ell_2)$.  Thus for $i\in\mathcal{I}(\ell_1)$, there exists an
$\varepsilon_i>0, $  such that $x^i+\Delta t Ax^i\in \mathcal{P}$, for every
$\Delta t\in[0,\varepsilon_i].$ For $j\in \mathcal{I}(\ell_2)$ there exists a $\hat\varepsilon_i>0$ such that  $\hat x^j+\Delta t A\hat x^j\in
\mathcal{P}$ for every $\Delta t\in[0, \hat\varepsilon_i]$. Let
$\hat{\tau}=\min\{\varepsilon_1,...,\varepsilon_{\ell_1},
\hat\varepsilon_1,...,\hat\varepsilon_{\ell_2}\}>0$. Then for every
$x_k\in \mathcal{P}$  there exist $\theta_i,
\hat\theta_j\geq0$ with $\sum_{i=1}^\ell\theta_i=1,$ such that  $x_k=\sum_{i=1}^{\ell_1} \theta_i
x^i+\sum_{j=1}^{\ell_2} \hat\theta_j \hat x^j$. Then we have $
x_{k+1}=x_k+\Delta t Ax_k=\sum_{i=1}^{\ell_1} \theta_i(x^i+\Delta t
Ax^i)+\sum_{j=1}^{\ell_2} \hat\theta_j(\hat x^j+\Delta t A\hat
x^j)\in \mathcal{P} $ for every $ \Delta t\in[0,\hat{\tau}]. $
\end{proof}

\begin{corollary}\label{cora1}
Assume that a polyhedral cone $\mathcal{C_P}$ is an invariant set
for the continuous system (\ref{dyna1}). Then there exists a
$\hat\tau>0$, such that for every $x_k\in \mathcal{C_P}$ and $ \Delta
t\in[0, \hat{\tau}]$, we have $x_{k+1} \in \mathcal{C_P}$, where
$x_{k+1}$ is obtained by the forward Euler method (\ref{euler1}).
\end{corollary}

We now consider the polyhedron and the backward Euler method. Note that a similar result can be found in \cite{horvathfarkasconf2006}. 
\begin{theorem}\label{thema311}
Assume that a polyhedron $\mathcal{P}$,  given as in  (\ref{poly1}),
is an invariant set for the continuous system (\ref{dyna1}).
Then there exists a $\hat\tau>0$, such that for every $x_k\in
\mathcal{P}$ and $\Delta t\in [0, \hat\tau]$, we have $x_{k+1}\in
\mathcal{P}$, where $x_{k+1}$ is obtained by the
backward Euler method (\ref{euler2}), i.e., $\mathcal{P}$ is an
invariant set for the discrete system.
\end{theorem}
\begin{proof}
Let
$\bar\tau=\sup\{\tau\,|\, I-A\Delta t \text{ is nonsingular for every } \Delta t\in [0,\tau]\}
$, and denote the relative interior  and the relative boundary\footnote{A point $x\in \mathcal{S}$ is called a relative interior point of $\mathcal{S}$ if $x$ is an interior point
of $\mathcal{S}$ relative to aff($\mathcal{S}$), where aff($\mathcal{S}$) is the smallest affine subspace containing $\mathcal{S}$. Then ri$(\mathcal{S})$ is defined as the set of all the relative interior points of $\mathcal{S}$, and $\text{rb}(\mathcal{S})$ is defined as $\text{cl}(\mathcal{S})\backslash \text{ri}(\mathcal{S})$ (see \cite{rock} p. 44). } of a set $\mathcal{S}$ by $\text{ri}(\mathcal{S})$ and $\text{rb}(\mathcal{S})$, respectively. Note that $\mathcal{P}$ is a closed set, thus for every $x_k\in \mathcal{P}$, one has either $x_k\in \text{ri}(\mathcal{P})$ or $x_k\in \text{rb}(\mathcal{P})$. We  consider the following two cases:

\emph{Case} 1). $x_k\in \text{ri}(\mathcal{P})$.
For every $\tau>0,$ we can reformulate $x_{k+1}=(I-A\Delta t)^{-1}x_k$ as $x_{k+1}+\frac{\Delta t}{\tau}x_{k+1}=x_k+\frac{\Delta t}{\tau}(x_{k+1}+\tau Ax_{k+1}),$ i.e.,
\begin{equation}\label{eq633}
x_{k+1}=\tfrac{\tau}{\tau+\Delta t} x_{k}+\tfrac{\Delta t}{\tau+\Delta t}(x_{k+1}+\tau Ax_{k+1}).
\end{equation}
Note that $\frac{\tau}{\tau+\Delta t}+\frac{\Delta t}{\tau+\Delta t}=1,$ $\frac{\tau}{\tau+\Delta t}>0, $ and $\frac{\Delta t}{\tau+\Delta t}>0$, thus $x_{k+1}$ is a convex combination of $x_k$ and $\bar{x}=x_{k+1}+\tau Ax_{k+1}.$ Further, we observe that $\bar{x}$ is the vector obtained by applying the forward Euler method at $x_{k+1}$ with steplength $\tau.$

Now we are going to prove that $x_{k+1}\in \text{ri}(\mathcal{P})$ for every $\Delta t\in [0,\bar{\tau}).$ This proof is by contradiction. Let us assume that there exists a $\bar\tau_1\in [0,\bar\tau)$, such that $x_{k+1}=(I-A\bar\tau_1)x_k\in \text{rb}(\mathcal{P}).$ We now choose a $\tau>0$, which is not larger than the threshold given in Theorem \ref{thema1}, thus we have $\bar{x}\in\mathcal{P}$ and
\begin{equation}\label{eq63311}
x_{k+1}=\tfrac{\tau}{\tau+\bar\tau_1} x_{k}+\tfrac{\bar\tau_1}{\tau+\bar\tau_1}\bar{x},
\end{equation}
which, by noting that $x_k\in\text{ri}(\mathcal{P})$, implies that $x_{k+1}\in \text{ri}(\mathcal{P}).$ This contradicts to the assumption that $x_{k+1}\in \text{rb}(\mathcal{P}).$

\emph{Case} 2). $x_k\in\text{rb}(\mathcal{P}).$ There exists a $y\in \text{ri}(\mathcal{P})$, such that $\bar{x}_k^\epsilon=x_k+\epsilon y\in \text{ri}(\mathcal{P})$, for every $\epsilon\in(0,1).$ By a similar discussion as in Case 1), we have that $\bar{x}_{k+1}^\epsilon=(I-A\Delta t)\bar{x}_k^\epsilon\in\text{ri}(\mathcal{P}),$ for every $\Delta t\in [0,\bar\tau).$ By letting $\epsilon\rightarrow 0$, we have that $\bar{x}_{k+1}^\epsilon\rightarrow x_{k+1}\in \mathcal{P}.$

We prove that every
 $\hat\tau\in (0,\bar\tau)$ satisfies the theorem.
The proof is complete.
\end{proof}

\begin{remark}\label{rempoly}
The proof of Theorem \emph{\ref{thema311}} also quantifies the value of the invariance preserving uniform steplength threshold $\hat\tau$, i.e., $\hat\tau\in(0,\bar\tau),$
where $\bar\tau=\sup\{\tau\,|\, I-A\Delta t$  is nonsingular for every $ \Delta t\in [0,\tau]\}
$. 
\end{remark}

\begin{corollary}\label{thema311ab}
Assume that a polyhedral cone $\mathcal{C_P}$  
is an invariant set for the continuous system (\ref{dyna1}).
Then there exists a $\hat\tau>0$, such that for every $x_k\in
\mathcal{C_P}$ and $\Delta t\in [0, \hat\tau]$, we have $x_{k+1}\in
\mathcal{C_P}$, where $x_{k+1}$ is obtained by the
backward Euler method (\ref{euler2}), i.e., $\mathcal{C_P}$ is an
invariant set for the discrete system.
\end{corollary}

\begin{theorem}\label{thema3}
Assume that an ellipsoid $\mathcal{E}$, given as in (\ref{elli}),
is an invariant set for the continuous system (\ref{dyna1}). Then
there exists a $\hat\tau>0$, such that for every $x_k\in \mathcal{E}$
and $\Delta t\in[0, \hat\tau]$, we have $x_{k+1}\in \mathcal{E}$,
where $x_{k+1}$ is obtained by the backward Euler method
(\ref{euler2}), i.e., $\mathcal{E}$ is an invariant set for the discrete system.
\end{theorem}
\begin{proof}
In the backward Euler method, the coefficient matrix is $(I-A\Delta
t)^{-1}$, where $\Delta t$ is the steplength.  Given any $x_k\in
\mathcal{E}$, according to Lemma \ref{lemma21}, there exists a
$\tau(x_k)>0$, such that $x_{k+1}\in \rm{int}(\mathcal{E})$ for
every $\Delta t\in (0, \tau(x_k)].$  In our proof, we need to bound
the magnitude of the coefficient matrix $(I-A\Delta t)^{-1}$. We
consider the eigenvalues of $(I-A\Delta t)^{-1}$, which are
$(1-\lambda_i(A)\Delta t)^{-1},$ for $i=1,2,...,n.$ To bound
$(1-\lambda_i(A)\Delta t)^{-1}$, we need $|\lambda_i(A)\Delta t|<1.$
Note that any positive $\tau<\tau(x_k)$ is also a bound for $\Delta
t,$ thus, for example,  we can choose $0<\Delta
t\leq\hat\tau(x_k):=\min\{\tau(x_k),\frac{1}{2|\lambda_i(A)|}\}=
\min\{\tau(x_k),\frac{1}{2\rho(A)}\},$ where $\rho(A)$ is the spectral radius (see, e.g. 
\cite{horn}) of $A,$
which yields $|1-\lambda_i(A)\Delta t|\geq\frac{1}{2}$. Thus, we need to
have that $\|(I-A\Delta t)^{-1}\|$ is uniformly bounded by $2$ on
$\mathcal{E}$ for every $\Delta t\in (0,\hat\tau(x_k)].$

Since $x_{k+1}=(I-A\hat\tau(x_k))^{-1}x_k\in
\rm{int}(\mathcal{E})$, we can choose a positive $r(x_{k+1})$,
such that the open ball
$\delta(x_{k+1},r(x_{k+1}))\subset\rm{int}(\mathcal{E})$. It is
easy to verify that the open ball $\delta(x_k,
\frac{1}{2}r(x_{k+1}))$ is mapped into $\delta(x_{k+1},r(x_{k+1}))$
by the backward Euler method. This is because for
$\tilde{x}_k\in\delta(x_k,\frac{1}{2}r(x_{k+1}))$, we apply the
backward Euler method at $\tilde{x}_k$ with $\hat\tau(x_k)$ to yield
$\tilde{x}_{k+1}=(I-A\hat\tau(x_k))^{-1}\tilde{x}_k$. Then we have
\begin{equation*}
\|\tilde{x}_{k+1}-x_{k+1}\|\leq
\|(I-A\hat\tau(x_k))^{-1}\|\|\tilde{x}_k-x_k\|\leq
2\|\tilde{x}_k-{x}_k\|\leq r(x_{k+1}),
\end{equation*}
i.e., $\tilde{x}_{k+1}\in \delta(x_{k+1},r(x_{k+1}))\subset
\rm{int}(\mathcal{E}). $ Therefore, we have that $\hat\tau(x_k)$
is a uniform  bound for $\Delta t$ at every point in $\delta(x_k,
\frac{1}{2}r(x_{k+1})).$

Obviously,
$\cup_{x_k\in\mathcal{E}}\delta(x_k,\frac{1}{2}r(x_{k+1}))$ is an
open cover of the ellipsoid $\mathcal{E}$. Since $\mathcal{E}$ is a
compact set, according to \cite{rudin}, there
exists a finite subcover
$\cup_{k=1}^m\delta(x_k,\frac{1}{2}r(x_{k+1}))$ of $\mathcal{E}$. For each open
ball $\delta(x_k,\frac{1}{2}r(x_{k+1}))$, there is a
uniform  bound $\hat\tau(x_k)$, thus, we have that
$\hat\tau=\min_{k=1,...,m}\{\hat\tau(x_k)\}$ is an invariance preserving  uniform bound for
$\Delta t$ for the backward Euler method at every point in
$\mathcal{E}.$ The proof is complete.
\end{proof}

We now consider to quantify a uniform steplength threshold of the backward Euler method for invariance preserving for ellipsoids. We need some technical results.

\begin{lemma}\label{lem:bound4ellp1}
(\emph{\cite{horn}}) Let $M\succ0$ ($M\succeq0$, $M\prec0$, or $M\preceq0$) and $N$ be a
nonsingular matrix.  Then $N^TMN\succ0$ ($N^TMN\succeq0$,
$N^TMN\prec0$, or $N^TMN\preceq0$).
\end{lemma}

\begin{lemma}\label{lem:bound4ellp2}
If $Q\succ0, A^TQ+QA\preceq 0,$ then
\begin{itemize}
\itemsep -2 mm
  \item  for $P=QA$, we have  $x^TPx\leq0$ for every $ x\in
\mathbb{R}^n$.
  \item for every $t\geq0,$ $I-At$ is nonsingular.
\end{itemize}
\end{lemma}

\begin{proof}
For  $x\neq0,$ 
$
2x^TPx=2x^T(QA)x=x^T(A^TQ+QA)x\leq 0,
$
that proves the first part.

For the second part, since
$
I-At=I-tQ^{-1}P=Q^{-1}(Q-tP),
$
the singularity of $I-At$ is equivalent to that of $Q-tP.$ Assume that the latter one is singular. Then there exists an
$x\neq0$, such that $(Q-tP)x=0$. Then
$
0=x^T(Q-tP)x=x^TQx-tx^TPx>0,
$
where the last inequality is due to  $Q\succ0$ and the first part.
This is a contradiction, thus the proof is complete. 
\end{proof}

The following theorem presents a uniform invariance preserving   steplength threshold of the backward Euler method for ellipsoids. The form of the threshold coincides with the one for polyhedra given in Remark \ref{rempoly}. Further, the uniform steplength threshold is proved to  be $\infty.$

\begin{theorem}\label{thm:bound4ellp}
Assume that an ellipsoid $\mathcal{E}$ is an
invariant set for the continuous system (\ref{dyna1}). Let
$\bar\tau=\sup\{\tau\,|\, I-A\Delta t \text{ is nonsingular for every } \Delta t\in [0,\tau]\}
$. Then $\bar\tau=\infty$ and thus for every $x_k\in \mathcal{C}$ and $\Delta t\geq0$, we
have that $x_{k+1}\in \mathcal{E}$, where $x_{k+1}$ is obtained by the
backward Euler method (\ref{euler2}), i.e., $\mathcal{E}$ is an
invariant set for the discrete system.
\end{theorem}
\begin{proof}
According to \cite{song1}, we have that $\mathcal{E}$ is an invariant set for the discrete
and continuous systems  if and only if $A^TQA-Q\preceq0$ and $A^TQ+QA\preceq0$, respectively. Then by Lemma \ref{lem:bound4ellp2}, we have that $\hat\tau=\infty.$ It is
easy to see that the theorem is equivalent to that $A^TQ+QA\preceq0$ implies that 
\begin{equation}\label{rdeq12}
(I-tA)^{-T}Q(I-tA)^{-1}-Q\preceq0
\end{equation}
holds for every $t\geq0,$ According to Lemma \ref{lem:bound4ellp1}, to prove (\ref{rdeq12}) is equivalent to prove 
$
Q-(I-tA)^TQ(I-tA)\preceq0,
$
i.e., 
\begin{equation}\label{eq98620}
A^TQ+QA-tA^TQA\preceq0.
\end{equation} 
Since $Q\succ0$, we have $A^TQA\succeq0,$ thus (\ref{eq98620}) is true. The proof is complete. 
\end{proof}

By using an analogous discussion as the one presented in the proof of Theorem \ref{thm:bound4ellp}, one can show that other discretization methods, e.g.,  Pad\'e[1,1], Pad\'{e}[2,2], etc., see e.g., \cite{baker},  also allow some uniform  invariance preserving steplength thresholds.

To establish a uniform invariance preserving steplength threshold for the backward Euler method for the Lorenz cone
$\mathcal{C_L}$, we first consider the case when no
eigenvector of the coefficient matrix $A$ in (\ref{dyna1}) is on the
boundary of $\mathcal{C_L}$.

\begin{theorem}\label{thm123}
Assume that a Lorenz cone $\mathcal{C_L}$, given as in 
(\ref{ellicone}), is an invariant set for the continuous system
(\ref{dyna1}), and no eigenvector of the coefficient matrix
$A$ in (\ref{dyna1}) is on $\partial(\mathcal{C_L}).$ Then there exists
a $\hat\tau>0$, such that for every $x_k\in \mathcal{C_L}$ and $\Delta
t\in[0, \hat\tau]$, we have $x_{k+1}\in \mathcal{C_L}$, where
$x_{k+1}$ is obtained by the backward Euler method (\ref{euler2}),
i.e., $\mathcal{C_L}$ is an invariant set for the discrete system.
\end{theorem}
\begin{proof}
If $x_k=0$, then for every $\Delta t\geq0$ we have $x_{k+1}=0\in
\mathcal{C_L}.$ We now consider the case when $x_k\neq0$. Our
proof has two steps.

\emph{The first step} of the proof is considering a uniform bound for
$\Delta t$ on a base (see \cite{barvinok} page 66) of the Lorenz
cone $\mathcal{C_L}.$  For every $x_k\in \mathcal{C_L}$, to have
$x_{k+1}=(I-A\Delta t)^{-1}x_k\in \mathcal{C_L},$  we need to have
\begin{equation}\label{eqqq}
\begin{split}
x_{k+1}^TQx_{k+1}=&x_k^TQx_k+\Delta t x_k^T(A^TQ+QA)x_k\\ 
&+(\Delta
t)^2x_k^T(A^{2T}Q+A^TQA+QA^2)x_k+\cdots\leq 0.
\end{split}
\end{equation}
Let us take a hyperplane $\mathcal{H}$ such that $\mathcal{H}$ intersected with
$\mathcal{C_L}$ is a compact set
$0\notin\mathcal{C_L^+}=\mathcal{H}\cap\mathcal{C_L}$. In fact, $\mathcal{C_L^+}$
is a base\footnote{In
practice, a possible way to obtain a base can be chosen as follows: we
first take a hyperplane through the origin that intersects
$\mathcal{C}$ only by the origin. Then shift the hyperplane
to $x^*$, where $x^*$ is an interior point of $\mathcal{C}$.
The intersection of the shifted hyperplane and $\mathcal{C}$ is a base of
$\mathcal{C}$. The base of $\mathcal{C}$ is a compact set.}  of cone $\mathcal{C_L}$. Then
$\mathcal{C_L^+}$ is a base of the Lorenz cone $\mathcal{C_L}.$ For
every $x_k\in \mathcal{C_L^+}$, we consider the following four cases:

  \emph{Case} 1): In this case,
  $x_k\in\rm{int}(\mathcal{C_L})\cap\mathcal{C_L^+}$, thus we have
  $x_k^TQx_k<0.$ Consequently, due to (\ref{eqqq}), $x_{k+1}\in \rm{int}(\mathcal{C_L})$
  for sufficiently small $\Delta t$.

  \emph{Case} 2): In this case, $x_k\in\partial(\mathcal{C_L})\cap\mathcal{C_L^+}$, and
  $x_k^T(A^TQ+QA)x_k<0$, thus we have $x_k^TQx_k=0.$ Since the constant term
  is zero and the first order term is negative in (\ref{eqqq}),  we have
  $x_{k+1}\in \rm{int}(\mathcal{C_L})$ for sufficiently small $\Delta t$.

  \emph{Case} 3): In this case, $x_k\in\partial(\mathcal{C_L})\cap\mathcal{C_L^+}$,  $x_k^T(A^TQ+QA)x_k=0$,
  and
  $Ax_k\notin\partial(\mathcal{C_L})\cup(-\partial(\mathcal{C_L})),$
  thus we have $x_k^TQx_k=0,$ and $x_k^T(A^{2T}Q+A^TQA+QA^2)x_k<0$.
  The last inequality is due to the proof of Lemma \ref{lemma22}. Since the constant
  term is zero, the first order term is also zero, and the second order term is negative
   in (\ref{eqqq}),  we have
  $x_{k+1}\in \rm{int}(\mathcal{C_L})$ for sufficiently small $\Delta t$.

\emph{Case} 4): In this case,
$x_k\in\partial(\mathcal{C_L})\cap\mathcal{C_L^+}$,
$x_k^T(A^TQ+QA)x_k=0$, and
$Ax_k\in\partial(\mathcal{C_L})\cup(-\partial(\mathcal{C_L})).$
However, since $x_k$ is nonzero, we have seen in the proof of Lemma
\ref{lemma22} that in this case $x_k$ is an eigenvector of $A$. This
violates the assumption of this theorem, thus this case is
not possible.

Therefore, for every $x_k\in\mathcal{C_L^+},$ there exists a
$\tau(x_k)>0,$ such that $x_{k+1}\in\rm{int}(\mathcal{C_L})$ for
every $\Delta t\in (0,\tau(x_k)].$ Also, note that $\mathcal{C_L^+}$
is a compact set, thus, according to a similar argument as in the
proof of Theorem \ref{thema3}, we have a uniform bound for $\Delta
t$, denoted by $\hat\tau(\mathcal{C_L^+}),$ on $\mathcal{C_L+},$
such that for every $x_k\in \mathcal{C_L^+},$ we have $x_{k+1}\in
\mathcal{C_L}$ for every $\Delta t\in
[0,\hat\tau(\mathcal{C_L^+})].$

\emph{The second step} of the proof is extending the uniform  bound
of the steplength from $\mathcal{C_L^+}$ to $\mathcal{C_L}.$ Let
$0\neq x_k\in \mathcal{C_L}.$ Then, because $\mathcal{C_L^+}$ is a
base of $\mathcal{C_L},$ there exists a scalar $\gamma>0$, such that
$\gamma x_k= \tilde{x}_k\in \mathcal{C_L^+}.$ Then we have
\begin{equation*}
x_{k+1}=(I-A\Delta t)^{-1}x_k=(I-A\Delta t)^{-1}{\gamma^{-1}}
\tilde{x}_k={\gamma^{-1}} \tilde{x}_{k+1}.
\end{equation*}
Since $\tilde{x}_{k+1}\in \mathcal{C_L}$ for every $\Delta t\in
[0,\tau(\mathcal{C_L^+})],$  we have $x_{k+1}\in \mathcal{C_L}$, for
every $\Delta t\in [0,\tau(\mathcal{C_L^+})].$

Therefore, $\tau(\mathcal{C_L^+})$ is a uniform  bound for the
steplength $\Delta t$ for the backward Euler method at every point of
$\mathcal{C_L}$.
 The proof is complete.
\end{proof}

Now, in a more general setting, we consider the uniform invariance preserving steplength threshold on a general  proper cone for linear dynamical systems. 
\begin{definition}
[\emph{\cite{loewy}}] A convex cone $\mathcal{C}$ is called \textbf{proper} if it is nonempty, closed, and pointed.
\end{definition}

We recall
the concept of a matrix to be cross-positive on a proper cone, which
is first proposed by Schneider and Vidyasagar in \cite{schne}.

\begin{definition}
[\emph{\cite{schne}}] Let $\mathcal{C}\in\mathbb{R}^n$ be a proper cone and $\mathcal{C}^*$ be the dual cone\footnote{The dual of cone $\mathcal{C}$ is defined as $\mathcal{C}^*=\{y\,|\,y^Tx\geq0\text{ for all } x\in\mathcal{C}\}.$} of $\mathcal{C}$.
The matrix $M\in \mathbb{R}^{n\times n}$ is called
\textbf{cross-positive} on $\mathcal{C}$ if for all $x\in
\mathcal{C}, y\in\mathcal{C}^*$ with $x^Ty=0$,  the inequality $x^TMy\geq0$ holds.
\end{definition}

The properties of cross-positive matrices are thoroughly 
studied in \cite{schne}. The following lemma, which directly follows from
Theorem 2 and Lemma 6 in \cite{schne}, is useful in our analysis.

\begin{lemma}\label{lemma51}
(\emph{\cite{schne}})  Let $\mathcal{C}\in\mathbb{R}^n$ be a proper cone, and
denote the following two sets of matrices: $\Sigma_\mathcal{C}=\{M\,|\, M \text{ is cross-positive on }
\mathcal{C}\,\}$, and $ \Pi_\mathcal{C}=\{M\,|\, (M+\alpha I)
(\mathcal{C}\backslash \{0\})\subseteq\rm{int}(\mathcal{C}) \text{
for some }\alpha\geq0\,\}.$ Then the closure of $\Pi_\mathcal{C}$ is
$\Sigma_\mathcal{C}.$
\end{lemma}

\begin{lemma}\label{lemma52}
Let $\mathcal{C}\in\mathbb{R}^n$ be a proper cone, and denote 
$\Omega_\mathcal{C}=\{M\,|\, M\mathcal{C}\subseteq\mathcal{C}\}$.
Then $\Omega_\mathcal{C}$ is closed.
\end{lemma}
\begin{proof}
Let $\{M_i\}$ be a sequence of matrices in $\Omega_\mathcal{C}$,
such that $\lim_{i\rightarrow\infty}M_i=M.$ We choose an arbitrary
$x\in\mathcal{C}$. For every $i$, since
$M_i\mathcal{C}\subseteq\mathcal{C}$, we have $M_ix=y_i\in \mathcal{C}$. Since $\mathcal{C}$
is closed, we have
$Mx=\lim_{i\rightarrow\infty}M_ix=\lim_{i\rightarrow\infty}y_i=\bar{y}\in\mathcal{C}$.
The proof is complete.
\end{proof}

The existence of a uniform invariance preserving steplength threshold  for a proper
cone is presented in the following theorem.

\begin{theorem}\label{thm12345}
Assume that a proper cone $\mathcal{C}\in\mathbb{R}^n$ is an
invariant set for the continuous system (\ref{dyna1}). Then there
exists a $\hat\tau>0$, such that for every $x_k\in \mathcal{C}$ and
$\Delta t\in[0, \hat\tau]$, we have $x_{k+1}\in \mathcal{C}$, where
$x_{k+1}$ is obtained by the backward Euler method (\ref{euler2}),
i.e., $\mathcal{C}$ is an invariant set for the discrete system.
\end{theorem}

\begin{proof}
Since $\mathcal{C}$ is an invariant set for the continuous system,
we have $e^{At}\mathcal{C}\subseteq\mathcal{C}$ for every $t\geq0.$
According to Theorem 3 in \cite{schne}, this is equivalent to that
the coefficient matrix $A$ is cross-positive on $\mathcal{C}$. Then
by Lemma \ref{lemma51}, there exists a sequence of matrices
$\{A_i\}$, where $(A_i+\alpha_i I)\mathcal{C}\subseteq\mathcal{C}$ for
some $\alpha_i \geq0$, such that $\lim_{i\rightarrow\infty}A_i=A.$
For simplicity, we introduce the notation  $B_i=A_i+\alpha_i I$, then
$B_i\mathcal{C}\subseteq\mathcal{C}.$

Then we consider $(I-A\Delta t)^{-1}$, i.e., the coefficient matrix
of the discrete system obtained by using the backward Euler method. Let
$\bar\tau=\sup\{\tau\,|\, I-A\Delta t$  is nonsingular for every $ \Delta t\in [0,\tau]\}
$, then we have $\bar\tau>0.$ Since $\lim_{i\rightarrow\infty}A_i=A$ for every
$0<\epsilon_1<\epsilon_2<\bar\tau,$ there exists an integer $\bar{n}>0$,  such that
for every $i>\bar{n},$ we have $I-A_i\Delta t$ is nonsingular for $\Delta
t\in[0, \tau_i],$ where
$\tau_i\in(\bar\tau-\epsilon_2,\bar\tau-\epsilon_1).$ Since
$\{\tau_i\}_{i>\bar{n}}$ is bounded, it has a convergent subsequence
$\{\tau_{i^*}\}$, i.e.,
$\lim_{i^*\rightarrow\infty}\tau_{i^*}=\hat\tau\in[\bar\tau-\epsilon_2,\bar\tau-\epsilon_1].$
Thus, we have $0<\hat\tau<\bar\tau,$ and $I-A\Delta t$ is nonsingular
for $\Delta t\in[0, \hat\tau]$. For every $i^*$ we have
\begin{equation}\label{eq61}
(I-A_{i^*}\Delta t)^{-1}=((1+\alpha_{i^*}\Delta t)I-B_{i^*}\Delta
t)^{-1}=\tfrac{1}{1+\alpha_{i^*}\Delta
t}(I-B_{i^*}\tfrac{\Delta t}{1+\alpha_{i^*}\Delta
t})^{-1}.
\end{equation}
Since $B_{i^*}\mathcal{C}\subseteq\mathcal{C}$ and $\frac{\Delta
t}{1+\alpha_{i^*}\Delta t}>0$, we have
\begin{equation}\label{eq62}
(I-B_{i^*}\tfrac{\Delta t}{1+\alpha_{i^*}\Delta
t})^{-1}\mathcal{C}=(I+\tfrac{\Delta
t}{1+\alpha_{i^*}\Delta t}B_{i^*}+(\tfrac{\Delta
t}{1+\alpha_{i^*}\Delta
t})^2B_{i^*}^2+...)\mathcal{C}\subseteq\mathcal{C}.
\end{equation}
Since $1+\alpha_{i^*}\Delta t>0,$ by (\ref{eq61}) and (\ref{eq62}),
we have $(I-A_{i^*}\Delta t)^{-1}\mathcal{C}\subseteq\mathcal{C}$
for $\Delta t\in[0, \tau_{i^*}].$

Finally,  since $(I-A\Delta
t)^{-1}=\lim_{{i^*}\rightarrow\infty}(I-A_{i^*}\Delta t)^{-1}$,
according to Lemma \ref{lemma52}, we have $(I-A\Delta
t)^{-1}\mathcal{C}\subseteq\mathcal{C}$ for $\Delta
t\in[0,\hat\tau].$ The proof is complete.
\end{proof}

In fact, according to the proof of Theorem \ref{thm12345}, we can
also give the exact value of a uniform bound for the steplength for a
proper cone.
\begin{corollary}\label{thm123456}
Assume that a proper cone $\mathcal{C}\in\mathbb{R}^n$ is an
invariant set for the continuous system (\ref{dyna1}). Let
$\bar\tau=\sup\{\tau\,|\, I-A\Delta t \text{ is nonsingular for every } \Delta t\in [0,\tau]\}
$. Then for every $x_k\in \mathcal{C}$ and $\Delta t\in[0, \bar\tau)$, we
have $x_{k+1}\in \mathcal{C}$, where $x_{k+1}$ is obtained by the
backward Euler method (\ref{euler2}), i.e., $\mathcal{C}$ is an
invariant set for the discrete system.
\end{corollary}
\begin{proof}
For every $\Delta t\in[0,\bar\tau),$ we choose
$0<\epsilon_1<\epsilon_2<\bar\tau-\Delta t.$ Then, by an argument similar
to the proof of Theorem \ref{thm12345}, we have that
$\hat\tau\in[\bar\tau-\epsilon_2, \bar\tau-\epsilon_1]$ is a uniform bound
of the steplength. Note that $\Delta t<\bar\tau-\epsilon_2\leq\hat\tau,$
and we can choose $\epsilon_2>0$ sufficiently small, then the corollary is immediate.
\end{proof}

Let us take an example to illustrate Corollary \ref{thm123456}.
\begin{example}
Consider the cone $\mathcal{C}=\{(\xi,\eta)\;|\;\xi^2\leq
\eta^2,\eta\geq0\},$ and the  system
$\dot{\xi}=3\xi-\eta, \dot{\eta}=-\xi+3\eta.$
\end{example}

The solution of the system is $\xi(t)=\frac{1}{2}(\alpha
e^{2t}-\beta e^{4t}), \eta(t)=\frac{1}{2}(\alpha e^{2t}+\beta
e^{4t})$, where $\alpha, \beta$ depend on the initial condition.
Clearly, $\mathcal{C}$ is an invariant set for the system. It is
easy to compute $\tau=\sup\,\{\,\Delta t\,|\, I-A\Delta t \text{ is
nonsingular}\}=\frac{1}{4}.$ When the backward Euler method is
applied, we have $\xi_{k+1}=\gamma({(1-3\Delta t)\xi_k-\Delta
t\eta_k}), \eta_{k+1}=\gamma(-\Delta t\xi+(1-3\Delta t)\eta_k)$,
where $\gamma={((1-3\Delta t)^2+(\Delta t)^2)^{-1}}$. To ensure that
$\xi_{k+1}^2\leq \eta_{k+1}^2$, we let $(1-3\Delta t)^2-(\Delta
t)^2\geq0,$ which yields that $\Delta t\leq \frac{1}{4}$. Note that the other
solution that $\Delta t \geq
\frac{1}{2}$ is not applicable.

Since a Lorenz cone is a proper cone, the following corollary is
immediate.

\begin{corollary}\label{thm1234567}
Assume that a Lorenz cone $\mathcal{C_L}$, given as in  
(\ref{ellicone}), is an invariant set for the continuous system
(\ref{dyna1}).  Then there exists a $\hat\tau>0$ such that for every
$x_k\in \mathcal{C_L}$ and $\Delta t\in [0, \hat\tau]$ we have
$x_{k+1}\in \mathcal{C_L}$, where $x_{k+1}$ is obtained by the
backward Euler method (\ref{euler2}), i.e., $\mathcal{C_L}$ is an
invariant set for the discrete system. Moreover, $\hat\tau\in[0,\bar\tau)$, where $\bar\tau$ is given as in Corollary \ref{thm123456}.
\end{corollary}

\subsection{General Results for Uniform Steplength Threshold}

The property that the forward Euler method has a uniform invariance preserving steplength threshold for  plays
a significant role in the proof of Theorem \ref{thema311}, thus  we now generalize the conclusion to closed and convex sets.
By a similar proof of Theorem \ref{thema311}, the following theorem is immediate.

\begin{theorem}\label{thm21}
Let  $\mathcal{S}$ be a closed and convex set. Assume that  $\mathcal{S}$ is
an invariant set for the continuous system (\ref{dyna1}), and let $\bar\tau=\sup\{\tau\,|\, I-A\Delta t $ is nonsingular for every $ \Delta t\in [0,\tau]\}
$. Assume that there exists a $\tilde\tau>0$, such that
for every $x_k\in \mathcal{S}$ and $\Delta t\in [0,\tilde\tau]$, we have $x_{k}+\Delta t Ax_k\in \mathcal{S}$.
Then  for every $x_k\in
\mathcal{S}$  and $\Delta t\in [0, \bar\tau)$, we have $x_{k+1}\in
\mathcal{S}$, where $x_{k+1}$ is obtained by the
backward Euler method (\ref{euler2}), i.e., $\mathcal{S}$ is an
invariant set for the discrete system.
\end{theorem}

The compactness of an ellipsoid plays an important role in the proof
of Theorem \ref{thema3}. Now we generalize Theorem \ref{thema3} to
compact sets.
\begin{theorem}\label{themacor3}
Let a set $\mathcal{S}$, and a discretization method  $x_{k+1}=D(\Delta t)x_k$ be given. Assume that the following conditions
hold:
\begin{enumerate}
  \item The set $\mathcal{S}$ is a compact set.
  \item For every $x_k\in \mathcal{S}$, there exists a $\tau(x_k)>0$,
  such that $x_{k+1}\in \rm{int}(\mathcal{S})$ for every $\Delta t\in(0,\tau(x_k)].$
  \item There exists a $\tilde{\tau}>0,$ such that  $\|D(\Delta t)\|$
  is uniformly bounded for every $x\in \mathcal{S}$ and  $\Delta t\in [0,\tilde{\tau}]$.
\end{enumerate}
Then there exists a $\hat\tau>0$, such that for every $x_k\in
\mathcal{S}$ and $\Delta t\in[0,\hat\tau]$, we have $x_{k+1}\in
\mathcal{S}$, i.e., $\mathcal{S}$ is an invariant set for the
discrete system.
\end{theorem}
\begin{proof}
\begin{figure}[h]
    \centering
    \includegraphics[width=0.4\textwidth]{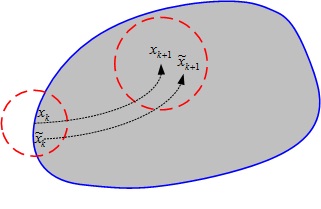}
    \caption{The idea of the proof of Theorem \ref{themacor3}.}
    \label{fig2}
\end{figure}
Note that every positive $\tau<\tau(x_k)$ is also a bound for $\Delta
t$ at $x_k$. Then let us define
$\hat{\tau}(x_k)=\min\{\tau(x_k),\tilde{\tau} \}$, according to
Condition 2, we have  $x_{k+1}=D(\hat{\tau}(x_k))x_{k}\in
\rm{int}(\mathcal{S})$. Thus we can choose an $r(x_{k+1})>0$, such
that the open ball
$\delta(x_{k+1},r(x_{k+1}))\subset\rm{int}(\mathcal{S}).$
According to Condition 3, there exists $0<K<\infty,$ such that
$\|D(\Delta t)\|\leq K$ for all $x\in \mathcal{S}$ and  $ \Delta t\in [0, \tilde{\tau}]$.

It is easy to verify that by the discretization method the open ball
$\delta(x_k, \frac{1}{K}r(x_{k+1}))$ is mapping into
$\delta(x_{k+1},r(x_{k+1}))$, see Figure \ref{fig2}. This is because for every
$\tilde{x}_k\in\delta(x_k,\frac{1}{K}r(x_{k+1}))$, the
discretization method applied to $\tilde{x}_k$ with steplength
$\hat\tau(x_k)$ yields
$\tilde{x}_{k+1}=D(\hat{\tau}(x_k))\tilde{x}_k$. Then we have
\begin{equation}\label{eqabond1}
\|\tilde{x}_{k+1}-x_{k+1}\|\leq
\|D(\hat{\tau}(x_k))\|\|\tilde{x}_k-x_k\|\leq
K\|\tilde{x}_k-{x}_k\|\leq r(x_{k+1}),
\end{equation}
i.e., $\tilde{x}_{k+1}\in \delta(x_{k+1},r(x_{k+1}))\subset
\rm{int}(\mathcal{S}). $ Therefore, we have that $\hat\tau(x_k)$
is a uniform bound for $\Delta t$ at every point in $\delta(x_k,
\frac{1}{K}r(x_{k+1})).$

Obviously,
$\cup_{x_k\in\mathcal{S}}\delta(x_k,\frac{1}{K}r(x_{k+1}))$ is an
open cover of  $\mathcal{S}$. Since $\mathcal{S}$ is a compact set, there
exists a finite subcover
$\cup_{k=1}^m\delta(x_k,\frac{1}{K}r(x_{k+1}))$ of $\mathcal{S}$. A
uniform bound for $\Delta t$ can be the smallest $\hat\tau(x_k)$ of
the finite number of  open balls
$\delta(x_k,\frac{1}{K}r(x_{k+1}))$, thus, we have that
$\hat\tau=\min_{k=1,...,m}\{\hat\tau(x_k)\}$ is a uniform bound for
$\Delta t$ for the discretization method at every point in
$\mathcal{S}.$ The proof is complete.
\end{proof}

\begin{remark}
  Assumption 2 in Theorem \ref{themacor3} implies that
  $0\notin\partial(S).$ Otherwise, for $x_k=0,$ we would have $x_{k+1}=D(\Delta t)0=0\in
  \partial(S)$ for every $\Delta t.$
\end{remark}

We can generalize Theorem \ref{themacor3} to nonlinear systems by introducing Lipschitz condition to replace Condition 3 in Theorem \ref{themacor3}. 

\begin{theorem}\label{themacor313}
Let a set $\mathcal{S}$, and a discretization method  $x_{k+1}=D(\Delta t,x_k)$ be given. Assume that the following conditions
hold:
\begin{enumerate}
  \item The set $\mathcal{S}$ is a compact set.
  \item For every $x_k\in \mathcal{S}$, there exists a $\tau(x_k)>0$,
  such that $x_{k+1}\in \rm{int}(\mathcal{S})$ for every $\Delta t\in(0,\tau(x_k)].$
  \item The Lipschitz condition holds for $D(\Delta t, x)$ with respect to $x$, i.e., there exists an $L>0,$ such that 
  \begin{equation}\label{lipcond}
  \|D(\Delta t, \tilde{x})-D(\Delta t, x)\|\leq L\|\tilde{x}-x\|, \text{ for } x,\tilde{x}\in \mathcal{S}.
  \end{equation}
\end{enumerate}
Then there exists a $\hat\tau>0$, such that for every $x_k\in
\mathcal{S}$ and $\Delta t\in[0,\hat\tau]$, we have $x_{k+1}\in
\mathcal{S}$, i.e., $\mathcal{S}$ is an invariant set for the
discrete system.
\end{theorem}
\begin{proof}
The proof is almost the same as the one presented in Theorem \ref{themacor3}. The only difference is that equation (\ref{eqabond1}) is replaced by
\begin{equation}\label{eqabond2}
\|\tilde{x}_{k+1}-x_{k+1}\|=
\|D(\Delta t, \tilde{x}_k)-D(\Delta t, x_k)\|\leq
L\|\tilde{x}_k-{x}_k\|\leq r(x_{k+1}),
\end{equation}
which is due to (\ref{lipcond}). 
\end{proof}

The assumption in Theorem \ref{thm123} that no eigenvector
of the coefficient matrix is on the boundary of $\mathcal{C_L}$
excludes the case that $x_{k+1}\in\partial(\mathcal{C_L}).$ We now
generalize Theorem \ref{thm123} to proper cones.

\begin{theorem}\label{themacor311}
Let a set $\mathcal{C}$, and a discretization method $x_{k+1}=D(\Delta
t)x_k$ be given. Assume that the following conditions hold:
\begin{enumerate}
\item The set $\mathcal{C}$ is a proper cone.
\item For every $0\neq x_k\in \mathcal{C}$, there exists
a $\tau(x_k)>0$, such that $x_{k+1}\in \rm{int}(\mathcal{C})$ for
every $\Delta t\in(0,\tau(x_k)].$
  \item There exists a $\tilde{\tau}>0,$ such that  $\|D(\Delta t)\|$
  is uniformly bounded for every $x\in \mathcal{S}$ and  $\Delta t\in [0,\tilde{\tau}]$.
\end{enumerate}
Then there exists a $\hat\tau>0$,
such that for every $x_k\in \mathcal{C}$ and $\Delta
t\in[0,\hat\tau]$, we have $x_{k+1}\in \mathcal{C}$, i.e.,
$\mathcal{C}$ is an invariant set for the discrete system.
\end{theorem}

\begin{proof}
If $x_k=0$ then for every $\Delta t\geq0$ we have $x_{k+1}=0\in
\mathcal{C}.$ We now consider the case when $x_k\neq0$. Our
proof has two steps.

\begin{figure}[h]
    \centering
    \includegraphics[width=0.3\textwidth]{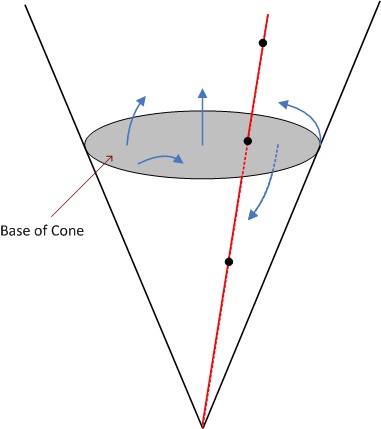}
    \caption{The idea of the proof of Theorem \ref{themacor311}.}
    \label{fig3}
\end{figure}

\emph{The first step} of the proof is considering a uniform bound for
$\Delta t$ on a base (see \cite{barvinok} page 66) of $\mathcal{C}.$
Since $\mathcal{C}$ is a proper cone, we can take a
hyperplane $\mathcal{H}$ to intersect it with $\mathcal{C}$ to
generate a base of $\mathcal{C}$ denoted by
$\mathcal{C^+}=\mathcal{H}\cap\mathcal{C}$. For every  $x_k\in
\mathcal{C^+},$ there exists a $\tau(x_k)>0$, such that $x_{k+1}\in
\rm{int}(\mathcal{C})$ for every $0<\Delta t\leq\tau(x_k).$ Note
that $\mathcal{H}$ and $\mathcal{C}$ are closed sets, thus
$\mathcal{C^+}$ is also a closed set. Also, assume that
$\mathcal{C^+}$ is unbounded, then $\mathcal{C}\cap(-\mathcal{C})$
contains a half line, which contradicts that $\mathcal{C}$ is a pointed
proper cone. Therefore, $\mathcal{C^+}$ is a compact set. Then,
according to a similar argument as in the proof of Theorem
\ref{themacor3}, we have a uniform bound for $\Delta t$, denoted by
$\hat\tau(\mathcal{C^+}),$ on $\mathcal{C^+},$ such that for every
$x_k\in \mathcal{C^+},$ we have $x_{k+1}\in \mathcal{C}$ for every
$\Delta t\in [0,\hat\tau(\mathcal{C^+})]$, see Figure \ref{fig3}.

\emph{The second step} of the proof is extending the uniform  bound
of the steplength $\Delta t$ from $\mathcal{C^+}$ to $\mathcal{C}.$
Let $0\neq x_k\in \mathcal{C}.$ Then, because $\mathcal{C^+}$ is a
base of $\mathcal{C}$, there exists a scalar $\gamma>0$ such that
$\gamma x_k= \tilde{x}_k\in \mathcal{C^+}.$ Then we have 
\begin{equation}\label{condcone}
x_{k+1}=D(\Delta t)x_k=D(\Delta t){\gamma^{-1}}
\tilde{x}_k={\gamma^{-1}} \tilde{x}_{k+1}. 
\end{equation} 
Since
$\tilde{x}_{k+1}\in \mathcal{C}$ for every $\Delta t\in
[0,\tau(\mathcal{C^+})],$  we have $x_{k+1}\in \mathcal{C}$ for
every $\Delta t\in [0,\tau(\mathcal{C^+})].$

Therefore, $\tau(\mathcal{C^+})$  is a uniform  bound for the
steplength $\Delta t$ for the discretization method at every point on
$\mathcal{C}$.
 The proof is complete.
\end{proof}

We can generalize Theorem \ref{themacor311} to nonlinear systems by adding a homogenous condition.  
\begin{theorem}\label{themacor312}
Let a set $\mathcal{C}$, and a discretization method $x_{k+1}=D(\Delta
t,x_k)$ be given. Assume that the following conditions hold:
\begin{enumerate}
\item The set $\mathcal{C}$ is a proper cone.
\item For every $0\neq x_k\in \mathcal{C}$, there exists
a $\tau(x_k)>0$, such that $x_{k+1}\in \rm{int}(\mathcal{C})$ for
every $\Delta t\in(0,\tau(x_k)].$
 \item The Lipschitz condition holds for $D(\Delta t, x)$ with respect to $x$, i.e., there exists an $L>0,$ such that 
  \begin{equation}\label{dlipcond}
  \|D(\Delta t, \tilde{x})-D(\Delta t, x)\|\leq L\|\tilde{x}-x\|, \text{ for } x,\tilde{x}\in \mathcal{S}.
  \end{equation}
\item The function $D(\Delta t, x)$ is homogeneous of degree $p\geq1$ with respect to $x$, i.e., 
\begin{equation}\label{homcond}
D(\Delta t, \alpha x)=\alpha^pD(\Delta t, x), \text{ for } \alpha\in \mathbb{R},\, x\in \mathcal{C}.
\end{equation}
\end{enumerate}
Then there exists a $\hat\tau>0$,
such that for every $x_k\in \mathcal{C}$ and $\Delta
t\in[0,\hat\tau]$ we have $x_{k+1}\in \mathcal{C}$, i.e.,
$\mathcal{C}$ is an invariant set for the discrete system.
\end{theorem}
\begin{proof}
The proof is almost the same as the one presented in Theorem \ref{themacor311}. The only difference is that equation (\ref{condcone}) will be replaced by
\begin{equation}\label{condcone1}
x_{k+1}=D(\Delta t,x_k)=D(\Delta t,{\gamma^{-1}}
\tilde{x}_k)={\gamma^{-p}} \tilde{x}_{k+1},
\end{equation} 
which is due to (\ref{homcond}).
\end{proof}

\section{Conclusions}\label{sec:con}
Invariant sets plays an important role both in the theory and practical applications of dynamical systems and control theory. A key topic in the study of this field is to investigate and derive conditions for discretization methods, and discretization steplength, so that an invariant set of a continuous system is also an invariant set for the corresponding discrete system obtained by using the discretization method. This problem can be referred to as to finding local or uniform invariance preserving steplength thresholds for the discretization methods. 

Existing results usually rely on the assumption that the explicit Euler method has an invariance preserving steplength threshold. In this paper, first we study the existence and the quantification of local and uniform invariance preserving steplength threshold for Euler methods on special sets, namely, polyhedra, ellipsoids, or Lorenz cones.  Our  novel proofs are using only elementary concepts.
We also extend our results and proofs to general convex sets, compact sets, and proper cones when a general discretization method is applied to linear or nonlinear dynamical systems. Conditions for the existence of a uniform invariance preserving steplegnth threshold for discretization methods on these sets are presented. This paper contributes to the study of invariant sets both in theory and in practice. 
One can use our results as  criteria  to check if a discretization method is invariance preserving with a uniform steplength threshold.
Once the existence of a uniform invariance preserving steplength threshold is ensured by the results of this paper, there is a need to find the optimal  steplength threshold for a given discretization method.  This will remain the subject of  future research.

\section*{Acknowledgements}

This research is supported by a Start-up grant of Lehigh University, and by
TAMOP-4.2.2.A-11/1KONV-2012-0012: Basic research for the development of
hybrid and electric vehicles. The TAMOP Project is supported by the European Union
and co-financed by the European Regional Development Fund.

\bibliographystyle{IMANUM-BIB}
\bibliography{myref}




\end{document}